\documentclass[11pt]{article} 
\usepackage{amsfonts,enumitem,amsmath,latexsym,amssymb,mathrsfs,amsthm,color,
comment}
\usepackage[affil-it]{authblk}
\usepackage{comment}

\let\OLDthebibliography\thebibliography
\renewcommand\thebibliography[1]{
  \OLDthebibliography{#1}
  \setlength{\parskip}{0pt}
  \setlength{\itemsep}{0pt plus 0.0ex}
}

\def\numberlikeadb{\global\def\theequation{\thesection.\arabic{equation}}}
\numberlikeadb
\newtheorem{theorem}{Theorem}[section]

\newcommand{\eq}{\begin{equation}}
\newcommand{\qe}{\end{equation}}

\def\N{{\rm I\kern-0.16em N}}
\def\R{{\rm I\kern-0.16em R}}
\def\E{{\rm I\kern-0.16em E}}
\def\P{{\rm I\kern-0.16em P}}
\def\F{{\rm I\kern-0.16em F}}
\def\B{{\rm I\kern-0.16em B}}
\def\Z{{\rm I\kern-0.46em Z}}
\def\C{{\rm I\kern-0.46em C}}
\def\G{{\rm I\kern-0.50em G}}

\begin{document}

\title{Stein's Method for the Single Server Queue in Heavy Traffic}
\author{Robert E. Gaunt\footnote{School of Mathematics, The University of Manchester, Manchester M13 9PL, UK}\, and
Neil Walto$\mathrm{n}^*$  
}
\date{} 
\maketitle

\vspace{-12mm}

\begin{abstract}Following recent developments in the application of Stein's method in queueing theory, this paper is intended to be a short treatment showing how Stein's method can be developed and applied to the single server queue in heavy traffic.  Here we provide two approaches to this approximation: one based on equilibrium couplings and  another involving comparison of generators.
\end{abstract}


\noindent{{{\bf{AMS 2010 Subject Classification:}}} Primary 60K25; Secondary 90B20; 60F99

\section{Introduction}

The M/G/1 and G/G/1 are queueing models of a single server with infinite buffer experiencing  arrivals of independent identically distributed jobs. The study of these classical queueing models was initiated by the work of Erlang \cite{erlang1909theory}, Pollaczek \cite{pollaczek1930aufgabe,Pollaczek1957} and Khinechine \cite{khinchin2013mathematical}. The limiting asymptotic where the load on the server approaches a critical level is know as Heavy Traffic. In this asymptotic, the rescaled waiting time approaches an exponential random variable, and as first noted by Kingman \cite{kingman}, the exponential distribution provides an appropriate approximation for waiting time in these queueing models.

Stein's method, as introduced by Charles Stein \cite{stein}, is a well established method for ascertaining the accuracy of approximation between two probability distributions.
The method, as detailed in Stein \cite{stein2}, consists of three key steps: first, a characterising equation for the target distribution is established, which leads to the so-called \emph{Stein equation}; second, appropriate bounds must be found for the solution of the Stein equation; thirdly, through a combination of the first two key ingredients and coupling techniques the error between the prelimit and target distributions is bounded in a certain probability metric.  Over the years, a number of different approaches to distributional approximations for numerous target distributions have been established in the Stein's method literature; an overview can be found in the survey Ross \cite{ross}.  We summarise and apply two such approaches to the stationary single server queue.

%
%

Stein's method has found applicability in a number of areas such as random graphs \cite{bhj92}, branching processes \cite{pekoz1} and statistical mechanics \cite{el10}; see Ross \cite{ross} for a recent review of applications and methods.
However, only recently has Stein's method begun to be applied to queueing theory.
Specifically, following the work of Gurvich \cite{g14}, Braveman and Dai in a series of papers and another together with Feng ascertained and developed the application of Stein's method in queueing using a Basic Adjoint Relation (BAR) approach \cite{bj1,braverman18,bj2,bjf}. These works principally provide approximations between Erlang queueing models and their limiting stationary distributions in the Halfin-Whitt asymptotic. As noted above, another limiting regime is Heavy Traffic. Braverman, Dai and Miyazawa \cite{bjm} apply their BAR approach to prove weak convergence of stationary distributions in Heavy Traffic.  Recently, Besan\c{c}on, Decreusefond, and Moyal \cite{bdm18} have used Stein's method to obtain explicit bounds for the diffusion approximations for the number of customers in the M/M/1 and M/M/$\infty$ queues.  Their results, which are obtained using the functional Stein's method introduced
for the Brownian approximation of Poisson processes \cite{cd13}, differ from the aforementioned results as they are given at the process level.  
More recently, Huang and Gurvich \cite{hg18} provide moment bounds using for a range of queueing models with abondonment; moment bounds for the M/G/1 queue are a special case of their analysis. Theorem 1.1 below provides upper and lower-bounds in the Wasserstein metric, a metric commonly applied to compare probability distributions. Here we gain new results
on the convergence of M/G/1 and G/G/1 queues using and developing Stein's machinery.
Consequently, we devise an efficient proof where numerical constants can be specified and
a further tighter bound is given under an alternative heavy traffic scaling.


 From the famous Pollaczek-Khinchine formula for moment-generating functions or via ladder-height arguments, it can be shown that the stationary waiting time distribution of the M/G/1 and G/G/1 queues can be expressed as a geometric convolution, that is the sum of a geometrically distributed number of IID random variables. 
%

We review a number of works that consider exponential approximations or geometric convolutions.  
The work of Brown \cite{brown1990error} finds approximations of geometric convolutions to the exponential distribution using renewal theory techniques, rather than Stein's method.
More recent work of Brown \cite{brown2015sharp} improves upon these bounds under certain hazard rate assumptions.
Recent works of Pek\"oz and R\"ollin \cite{pekoz1} and Pek\"oz, R\"ollin and Ross \cite{prr} apply Stein's method to the exponential and geometric approximations respectively. Theorem \ref{thmross}, below, is analogous to Theorem 3.1 of \cite{pekoz1}. Contemporaneously with the work of Braverman and Dai, Daly \cite{daly} applies Stein's method to quantify the approximation between geometric convolutions and non-negative integer valued random variables.

Consider now the M/G/1 queue with inter-arrival times following the $\mathrm{Exp}(\lambda)$ distribution and a general service time distribution $S$.  We let $W$ denote its stationary waiting time and define $\rho= \lambda \mathbb E [S]$ to be its load.
It is a well-known result \cite{kingman} that the stationary waiting time of a M/G/1 queue is approximately exponentially distributed in the heavy traffic limit.  Specifically,
\[
(1-\rho ) W \Rightarrow \left( \frac{\mathbb E [S^2]}{2 \mathbb E [S]}\right) Z,
\quad 
\text{as $\rho \rightarrow 1$},
\]
where $Z$ is an exponential parameter $1$ random variable, $\Rightarrow $ denotes weak convergence, and here and throughout the paper we assume $S$ has finite third moment.
We quantify this approximation by providing bounds in the Wasserstein metric, which, for non-negative random variables $U$ and $V$, is defined to be
\begin{equation*}
d_{\mathrm{W}}(\mathcal{L}(U),\mathcal{L}(V))=\sup_{h\in\mathrm{Lip}(1)}|\mathbb{E}[h(U)]-\mathbb{E}[h(V)]|,
\end{equation*}
where $\mathrm{Lip}(1)=\{h:\mathbb{R}^+\rightarrow\mathbb{R}\,:\,|h(x)-h(y)|\leq|x-y|,\:\forall\, x,y\geq0\}$.  Letting $F$ and $G$ denote the distribution functions of $U$ and $V$ respectively, we have the equivalent defintion (see Gibbs and Su \cite{gs02}):
\begin{equation}\label{wass2}
d_{\mathrm{W}}(\mathcal{L}(U),\mathcal{L}(V))=\int_0^\infty|F(x)-G(x)|\,\mathrm{d}x.
\end{equation}

In the theorem below we provide approximations for two scalings of the waiting time of an M/G/1 queue.

\begin{theorem}\label{thmbound}
For the stationary M/G/1 described above, let 
\begin{equation*}
	\hat{W}=\frac{2\mathbb{E}[S]}{\mathbb{E}[S^2]}(1-\rho)W  
	\quad
	\text{and}
	\quad 
	\widetilde{W}=\frac{1}{\rho}\hat{W}.
\end{equation*}
Then,
\begin{equation}\label{mg1bound}
d_{\mathrm{W}}(\mathcal{L}(\widetilde{W}),\mathcal{L}(Z))\leq\frac{4\mathbb{E}[S^3]\mathbb{E}[S]}{3(\mathbb{E}[S^2])^2}\frac{1-\rho}{\rho},
\end{equation}
and
\begin{equation}\label{mg1bound2}1-\rho\leq d_{\mathrm{W}}(\mathcal{L}(\hat{W}),\mathcal{L}(Z))\leq \bigg(1+\frac{4\mathbb{E}[S^3]\mathbb{E}[S]}{3(\mathbb{E}[S^2])^2}\bigg)(1-\rho).
\end{equation}
The $O(1-\rho)$ rate as $\rho\rightarrow1$ in (\ref{mg1bound}) is optimal.
\end{theorem}

In the literature, the normalisation $\hat{W}$ of the stationary waiting time distribution is a more common heavy traffic scaling than the normalisation $\widetilde{W}$.  However, because the expectation of $\hat{W}$ is not equal to that of the $\mathrm{Exp}(1)$ distribution, we have a larger error in bound (\ref{mg1bound2}) than bound (\ref{mg1bound}).  In proving Theorem \ref{thmbound} we shall first establish the bound (\ref{mg1bound}) (in which $\widetilde{W}$ and $Z$ have the same mean) and then deduce (\ref{mg1bound2}) as a simple consequence.  Theorem \ref{thmbound} includes a lower bound and a statement regarding the optimality of the rate in the bounds.  Such results are not commonly found in the Stein's method literature and are therefore of interest, even if the method of proof is not new to this paper.

We provide two proofs of Theorem \ref{thmbound}.  Rather curiously, two quite different approaches result in exactly the same upper bound (\ref{mg1bound}). One proof analyses the generator of the M/G/1 queue and compares this to the Stein equation of the limiting exponential, as such this generator approach to Stein's method is similar to the BAR method used in \cite{braverman18,bj2,bj1,bjf}.
Prior works applying Stein's method to queueing have typically considered phase-type job size distributions. We note that the results found here hold for general job size distributions.
The other proof applies an equilibrium coupling approach. This is the first time that a coupling approach to Stein's method has been used in the context of queueing theory, and allows us to analyse the G/G/1 queue. To the best of our knowledge, general arrivals have not been proven; only Markovian results using comparison of generators.  In our proof, we note that both the M/G/1 and G/G/1 queue have a stationary distribution that is the convolution of a geometrically distributed number of IID random variables, and
we prove a variant of results in Pek\"oz and R\"ollin \cite{pekoz1} and Ross \cite{ross}. 
 In addition, we provide a new result following Gaunt \cite{gaunt vg2} which proves that the rate of convergence considered is optimal. From our results we can 
deduce the following bound for the G/G/1 queue. 
(More detail on the terms in the bound below will be provided in Section \ref{sec2}.)

\begin{theorem}\label{gg1thm}Let $W$ be the stationary waiting time distribution of the G/G/1 queue.  Let 
$$\widetilde{W}=\frac{1-\eta}{\eta\mathbb{E}[Y_1]}W.$$ 
Where $Y_1$ is the first ladder height and $\eta$ is the probability of a finite ladder epoch for the random walk determining the evolution of the G/G/1 queue.
Then,
\begin{equation}\label{gg1bound}d_{\mathrm{W}}(\mathcal{L}(\widetilde{W}),\mathcal{L}(Z))\leq \frac{\mathbb{E}[Y_1^2]}{\big(\mathbb{E}[Y_1]\big)^2}\frac{1-\eta}{\eta}.
\end{equation}
The $O(1-\eta)$ rate as $\eta\rightarrow1$ in (\ref{gg1bound}) is optimal.
\end{theorem}

The bound of Theorem \ref{gg1thm} is in a sense less explicit than those of Theorem \ref{thmbound}, being given in terms of the more involved quantities $\eta$, $\mathbb{E}[Y_1]$ and $\mathbb{E}[Y_1^2]$.  Standard results for
the G/G/1 queue are stated in terms of $\eta$ and expectations involving $Y_1$; see, for example,
several results in Section 11.5 of Grimmett and Stirzaker \cite{gsbook}. Further, textbook of Feller
(Chapter XII Section 3) specifies the distribution of $Y_1$ using Wiener-Hopf Factorisation.
However, in some cases these quantities can be computed explicitly. This is the case for
the M/G/1 queue (see Section \ref{sec4} for details), and thus the bound (\ref{mg1bound}) of Theorem 1.1
can be obtained directly from the bound (\ref{gg1bound}) of Theorem \ref{gg1thm}.

As evidenced by, for example, Braverman and Dai \cite{bj1} it is possible to use Stein's method to obtain bounds in metrics other than the Wasserstein distance in distributional approximations that arise in queueing theory.  Also, Proposition 1.2 of Ross \cite{ross} can be used to immediately translate the Wasserstein distance bounds of Theorems \ref{thmbound} and \ref{gg1thm} into Kolmogorov distance bounds, although the resulting bounds have sub-optimal rate of convergence.  In this paper, we restrict our attention to the Wasserstein metric because it is very natural in the context of Stein's method and allows for a simple and clear exposition that would not be possible if working with the Kolmogorov metric.  Moreover, an accurate Kolmogorov distance bound can be readily obtained from a general result of Brown \cite{brown2015sharp} that concerns Kolmogorov error bounds for the exponential approximation of geometric convolutions.  The following bound is obtained from combining the final inequality on p$.$ 846 of \cite{brown2015sharp} and the representation (\ref{wrep1}) for the waiting time distribution of the G/G/1 queue in steady state: 
\begin{equation*}\label{brownk}
d_{\mathrm{K}}(\mathcal{L}(\widetilde{W}),\mathcal{L}(Z))\leq 1-\exp\bigg(-\frac{\mathbb{E}[Y_1^2]}{2\big(\mathbb{E}[Y_1]\big)^2}\frac{1-\eta}{\eta}\bigg)\leq \frac{\mathbb{E}[Y_1^2]}{2\big(\mathbb{E}[Y_1]\big)^2}\frac{1-\eta}{\eta}.
\end{equation*}
This matches the order we find for the Wasserstein metric via an incisive application of Stein's method.

The rest of this paper is organised as follows. In Section \ref{sec2}, we recall several classical results about M/G/1 and G/G/1 queues that shall be needed in the sequel.  In Section \ref{sec3.1}, we give an overview of Stein's method for exponential approximation.  In Section \ref{sec3.2}, we consider the equilibrium coupling approach and give a general Wasserstein distance bound for the exponential approximation of geometric convolutions.  In Sections \ref{sec4} and \ref{sec5}, respectively, we use the equilibrium coupling and comparison of generators approaches to prove Theorem \ref{thmbound}.

\section{Properties of the M/G/1 and G/G/1 queues}\label{sec2}

Here we collect together several know results about the M/G/1 and G/G/1 queues.
A thorough analysis of the M/G/1 queue can be found in 
Kleinrock \cite{Kleinrock}, Chapter 5. A ladder-process analysis of the G/G/1 queue can be found in Asmussen \cite{Asm}, Chapters VIII and X.

First consider the G/G/1 queue. 
Let $S_i$ be the service time of the $i$th customer and $X_i$ be the length of time between the $i$th and $(i+1)$th arrivals. Let $U_i=S_i-X_{i+1}$ and define
\begin{equation*}
\Sigma_0=0, \quad \Sigma_i=\sum_{j=1}^i U_j, \quad n\geq1.
\end{equation*}
It is well-known that, as a consequence of Lindley's recursion, the stationary waiting time of the G/G/1 queue is given by
\[
W= \max_{i\in \mathbb Z_+} \left\{ \Sigma_i \right\} .
\]

Define an increasing sequence $L(0),L(1),\ldots$ of random variables by
\begin{equation*}L(0)=0,\quad L(n+1)=\min\big\{ i >L(n)\,:\,\Sigma_i>\Sigma_{L(n)}\big\};
\end{equation*}
that is, $L(n+1)$ is the earliest epoch $i$ of time at which $\Sigma_i$ exceeds the random walk's previous maximum $\Sigma_{L(n)}$.  The $L(n)$ are called ladder times. Here 
\[
\eta=\mathbb{P}(\Sigma_n>0\;\text{for some $n\geq1$})
\]
 is the probability that at least one ladder point exists.  The total number $\Lambda$ of ladder points follows the $\mathrm{Geo}^0(1-\eta)$ distribution with probability mass function $\mathbb{P}(\Lambda=n)=(1-\eta)\eta^{n}$, $n=0,1,2,\ldots$.  Let
\begin{equation*}Y_j=\Sigma_{L(j)}-\Sigma_{L(j-1)}
\end{equation*}
be the difference in the displacements of the walk at the $(j-1)th$ and $j$th ladder points.  Conditional on the value of $\Lambda$, $\{Y_j\,:\,1\leq j\leq\Lambda\}$ is a collection of IID random variables. 
Furthermore,
\begin{equation}\label{wrep1}W=\Sigma_{L(\Lambda)}=\sum_{j=1}^\Lambda Y_j.
\end{equation}
Thus we note that the waiting time of a G/G/1 queue is the sum of a geometrically distributed number of IID random variables.

We now turn our attention to the M/G/1 queue. It can be seen that the infinitesimal generator of the waiting time process $\{W_t, \:t\geq0\}$ of a M/G/1 queue is given by 
\begin{equation*}
G_{W_t}g(y)=\lambda\int_0^\infty[g(y+s)-g(y)]\,\mathrm{d}F(s)-g'(y)\mathbf{1}(y>0),
\end{equation*}
where $F$ denotes the distribution function of $S$.  Here the integral term accounts for the jumps due to the arrival of work and the derivative term corresponds to the downward drift due to service. Let 
\begin{equation*}\delta=\frac{2\mathbb{E}[S]}{\mathbb{E}[S^2]}\frac{1-\rho}{\rho}.
\end{equation*}
Then rescaling $x=\delta y$ and substituting $g(x)=f(x/\delta)$ gives the following generator for the normalised waiting time process $\widetilde{W}_t=\delta W_t$:
\begin{equation}\label{wgen}G_{\widetilde{W}_t}f(x)=\lambda\int_0^\infty[f(x+\delta s)-f(x)]\,\mathrm{d}F(s)-\delta f'(x)\mathbf{1}(x>0).
\end{equation}
When $\widetilde{W}_t$ is stationary, we have that 
\begin{equation}\label{genegn9}
\mathbb{E}[G_{\widetilde{W}_t}f(\widetilde{W})]=0,
\end{equation}
for $f:\mathbb{R}^+\rightarrow\mathbb{R}$ a once continuously differentiable function.
This can be shown via a Fourier analysis applied to an integro-differential equation derived from \eqref{genegn9}. We refer the reader to Takacs \cite{Tak55}  and Section 5.12 of Kleinrock \cite{Kleinrock} for details.

\section{Stein's method for exponential approximation}\label{sec3}

In this section, we present results from Stein's method for exponential approximation that will be used to obtain Wasserstein distance bounds for the exponential approximation of the stationary waiting time distribution of the M/G/1 and G/G/1 queues.  Our treatment follows that of Pek\"oz and R\"ollin \cite{pekoz1} and Ross \cite{ross}; alternative approaches can be found in Chatterjee, Fulman and R\"ollin \cite{cfr11}. 

\subsection{The exponential Stein equation}\label{sec3.1}
Firstly, we briefly review the characterisation  
which can be found in Stein et al$.$ \cite{stein3}. This lies at the heart of Stein's method for exponential approximation.  The random variable $Z$ has the $\mathrm{Exp}(1)$ distribution if and only if
\begin{equation}\label{expchar} \mathbb{E}[f''(Z)-f'(Z)+f'(0)]=0
\end{equation}
for all functions $f:\mathbb{R}^+\rightarrow\mathbb{R}$ with Lipschitz derivative. (Usually, the characterisation is given in terms of $g=f'$; we shall see in Section \ref{sec5} why it is helpful to introduce an extra derivative.)  The characterising equation (\ref{expchar}) leads to the so-called \emph{Stein equation}:
\begin{equation}\label{expsteineqn}f_h''(x)-f_h'(x)+f_h'(0)=h(x)-\mathbb{E}[h(Z)],
\end{equation}
where $h:\mathbb{R}^+\rightarrow\mathbb{R}$ is a test function and $Z\sim\mathrm{Exp}(1)$.  The unique solution of (\ref{expsteineqn}) such that $f_h'(0)=0$ is given by
\begin{equation}\label{steinsoln}f_h'(x)=-\mathrm{e}^x\int_x^\infty\big(h(t)-\mathbb{E}[h(Z)]\big)\mathrm{e}^{-t}\,\mathrm{d}t.
\end{equation}
If $h$ is Lipschitz then the third derivative
 of $f_h$ satisfies the following bound (see \cite{pekoz1}, Lemma 4.1): 
\begin{align}\label{fbound1}\|f_h^{(3)}\|_\infty\leq 2\|h'\|_\infty.
\end{align} 

Now, evaluating both sides of (\ref{expsteineqn}) at a random variable of interest $W$ and taking expectations gives that
\begin{equation}\label{eqnh}|\mathbb{E}[h(W)]-\mathbb{E}[h(Z)]|=|\mathbb{E}[f_h''(W)-f_h'(W)]|.
\end{equation}
If, for example, we take the supremum of both sides of (\ref{eqnh}) over all functions $h$ from the class $\mathrm{Lip}(1)$, then bounding the quantity $d_{\mathrm{W}}(\mathcal{L}(W),\mathcal{L}(Z))$ reduces to bounding the right-hand side of (\ref{eqnh}) with the supremum taken over all $f_h$ for which $h\in\mathrm{Lip}(1)$.  In Section \ref{sec3.2}, we shall consider one approach to bounding the right-hand side of (\ref{eqnh}), which we shall make use of in Section \ref{sec4}.  Another is the comparison of generators approach that will be described and applied in Section \ref{sec5}. 
 



\subsection{The equilibrium coupling}\label{sec3.2}

We begin with a definition (see Pek\"oz and R\"ollin \cite{pekoz1}).  Let $W \geq 0$ be a random variable with $\mathbb{E}[W]<\infty$. We say that
$W^e$ has the \emph{equilibrium distribution with respect to $W$} if 
\begin{equation*}\mathbb{E}[f'(W)] - f'(0) = \mathbb{E}[W]\mathbb{E}[f''(W^e)] 
\end{equation*}
for all functions $f:\mathbb{R}^+\rightarrow\mathbb{R}$ with Lipschitz derivative.  For such random variables $W$, the equilibrium distribution exists and is given by $W^e=UW^s$, where $U\sim U(0,1)$ and $W^s$, the size bias distribution of $W$, are independent (see \cite{pekoz1}, Section 2.1.1). 
The size bias distribution of $W$ is given by $\mathrm{d}F_{W^s}(x)=x\mathrm{d}F_W(x)/\mathbb{E}[W]$.

Now suppose that $\mathbb{E}[W]=1$ and $\mathbb{E}[W^2]<\infty$.  If $W^e$ has the equilibrium distribution with respect to $W$, then 
\begin{align*}|\mathbb{E}[h(W)]-\mathbb{E}[h(Z)]|&=|\mathbb{E}[f_h''(W)-f_h'(W)]|\\
&=|\mathbb{E}[f_h''(W)-f_h''(W^e)]|\leq \|f_h^{(3)}\|_\infty\mathbb{E}|W-W^e|.
\end{align*} 
Applying (\ref{fbound1}) and then taking the supremum of both side over all functions $h$ from the class $\mathrm{Lip}(1)$ yields the following bound (see \cite{pekoz1}, Theorem 2.1):
\begin{equation}\label{ebound}d_{\mathrm{W}}(\mathcal{L}(W),\mathcal{L}(Z))\leq 2\mathbb{E}|W-W^e|.
\end{equation}

Recall from Section \ref{sec2} that the stationary waiting time distribution of the G/G/1 queue can be represented as a geometric convolution of the form $\sum_{j=1}^N Y_j$, where $N\sim\mathrm{Geo}^0(p)$ and the $Y_j$ are IID and independent of $N$.  In the following theorem, we use the bound (\ref{ebound}) to obtain a Wasserstein distance bound between such a geometric convolution normalised to have mean one and the $\mathrm{Exp}(1)$ distribution. 

 The proof of our bound follows very closely that of Theorem 3.1 of Pek\"oz and R\"ollin \cite{pekoz1} and Theorem 5.11 of Ross \cite{ross}.  Indeed, a special case of these theorems is to a geometric convolution where $N$ now follows the $\mathrm{Geo}(p)$ distribution  rather than the $\mathrm{Geo}^0(p)$ distribution (here the $\mathrm{Geo}(p)$ distribution has probability mass function $p(k)=p(1-p)^{k-1}$, $k=1,2,3,\ldots$, whilst the $\mathrm{Geo}^0(p)$ distribution has probability mass function $p(k)=p(1-p)^{k}$, $k=0,1,2,\ldots$).  However, it should be noted, that one cannot immediately translate the results of \cite{pekoz1} and \cite{ross} to the $\mathrm{Geo}^0(p)$ distribution; indeed, as observed by Pek\"oz, R\"ollin and Ross \cite{prr}, Stein's method for $\mathrm{Geo}^0(p)$ and $\mathrm{Geo}(p)$ approximation often has to be developed in parallel. 

We prove that the rate of convergence of our bound is optimal. This follows a recent argument used in the proof of Theorem 5.10 of Gaunt \cite{gaunt vg2}. As such this result was not found in \cite{pekoz1} and \cite{ross}; however, one can readily adapt our argument to show that the rates of convergence in the analogous results of \cite{pekoz1} and \cite{ross} are optimal. 


\begin{theorem}\label{thmross}Let $X_1,X_2,\ldots$ be IID random variables with $\mathbb{E}[X_1]=\mu$ and $\mathbb{E}[X_1^2]=\mu_2$.  Let $N\sim\mathrm{Geo}^0(p)$, and suppose that $N$ is independent of the $X_i$.  Set $W=\frac{p}{\mu(1-p)}\sum_{i=1}^NX_i$ and $Z\sim \mathrm{Exp}(1)$.  Then
\begin{equation}\label{rossbound}d_{\mathrm{W}}(\mathcal{L}(W),\mathcal{L}(Z))\leq \frac{p\mu_2}{(1-p)\mu^2}.
\end{equation}
Moreover, the $O(p)$ rate as $p\rightarrow0$ in (\ref{rossbound}) is optimal.
\end{theorem}

\begin{proof}For ease of notation, we prove the result for the case $\mu=1$; the generalisation to general $\mu>0$ is clear.  We begin by proving that 
\begin{equation}\label{wewe}W^e=\frac{p}{1-p}\bigg(\sum_{i=1}^NX_i+X_{N+1}^e\bigg)
\end{equation}
is an equilibrium coupling of $W$.  Let $f$ have a Lipschitz derivative with $f'(0)=0$ and define $g(m)=f'\big(\frac{p}{1-p}\sum_{i=1}^mX_i\big)$.  Using independence and the defining relation of $X_m^e$ gives that
\begin{equation*}\mathbb{E}\bigg[f''\bigg(\frac{p}{1-p}\sum_{i=1}^NX_i+\frac{p}{1-p}X_{N+1}^e\bigg)\,\bigg|\,N\bigg]=\frac{1-p}{p}\mathbb{E}[g(N+1)-g(N)\,|\,N].
\end{equation*}
We can use the formula $\mathbb{P}(N=n)=p\mathbb{P}(N\geq n)$ to obtain
\begin{equation*}\frac{1-p}{p}\mathbb{E}[g(N+1)-g(N)\,|\,(X_i)_{i\geq1}]=\mathbb{E}[g(N)\,|\,(X_i)_{i\geq1}].
\end{equation*}
Therefore, $\mathbb{E}[f''(W^e)]=\mathbb{E}[g(N)]=\mathbb{E}[f'(W)]$, as required.

Now, on substituting (\ref{wewe}) into (\ref{ebound}), we obtain
\begin{align*}d_{\mathrm{W}}(\mathcal{L}(W),\mathcal{L}(Z))&\leq 2\mathbb{E}|W-W^e|=\frac{2p}{1-p}\mathbb{E}[X_{N+1}^e] \\
&=\frac{2p}{1-p}\mathbb{E}[\mathbb{E}[X_{N+1}^e\,|\,N]]=\frac{p}{1-p}\mathbb{E}[\mathbb{E}[X_{N+1}^2\,|\,N]]=\frac{p\mu_2}{1-p}.
\end{align*}

We now prove that the $O(p)$ rate in (\ref{rossbound}) is optimal.  Consider the test function $h(x)=\cos(tx)$, $|t|\leq1$, which is in the class $\mathrm{Lip}(1)$.  Firstly, we record that 
\[\mathbb{E}[\cos(tZ)]=\int_0^\infty \cos(tx)\mathrm{e}^{-x}\,\mathrm{d}x=\frac{1}{1+t^2}.\]  
We now consider the characteristic function $\varphi_W(t)=\mathbb{E}[\mathrm{e}^{\mathrm{i}tW}]$, and note the relation $\mathbb{E}[\cos(tW)]=\mathrm{Re}[\varphi_W(t)]$.  From the above, $d_{\mathrm{W}}(\mathcal{L}(W),\mathcal{L}(Z))\geq |\mathrm{Re}[\varphi_W(t)]-\frac{1}{1+t^2}|$.  Recall that the probability-generating function of $N\sim \mathrm{Geo}^0(p)$ is given by $G_N(s)=\frac{p}{1-(1-p)s}$, $s<-\log(1-p)$.  Then 
\begin{equation}\label{refds}\varphi_W(t)=G_N\big(\varphi_{X_1}(\tfrac{pt}{\mu(1-p)})\big)=\frac{p}{1-(1-p)\varphi_{X_1}(\frac{pt}{\mu(1-p)})}.
\end{equation}
Now, since $\mathbb{E}[X_1]=\mu$ and $\mathbb{E}[X_1^2]=\mu_2$, as $p\rightarrow0$,
\begin{equation}\label{refgh}\varphi_{X_1}\big(\tfrac{pt}{\mu(1-p)}\big)=1+\frac{\mathrm{i}pt}{1-p}-\frac{1}{2}\frac{p^2t^2}{(1-p)^2}\frac{\mu_2}{\mu^2} +O(p^3).
\end{equation}
Substituting (\ref{refgh}) into (\ref{refds}) and performing an asymptotic analysis using the formula $\frac{1}{1+z}=1-z+O(|z|^2)$, $|z|\rightarrow0$, gives that, as $p\rightarrow0$, 
\begin{align*}\varphi_W(t)&=\frac{1}{\displaystyle 1-\mathrm{i}t+\frac{\mu_2}{2\mu^2}\frac{p^2t^2}{1-p}+O(p^2)}=\frac{1}{\displaystyle 1-\mathrm{i}t+\frac{\mu_2}{2\mu^2}p^2t^2+O(p^2)} \\
&=\frac{1}{1-\mathrm{i}t}\bigg(1-\frac{\mu_2}{2\mu^2}\frac{p^2t^2}{1-\mathrm{i}t}\bigg)+O(p^2) \\
&=\frac{1+2\mathrm{i}t-t^2}{2\mu^2(1+t^2)^2}\bigg(2\mu^2-\frac{\mu_2}{2\mu^2}p^2t^2-2\mathrm{i}\mu^2t\bigg)+O(p^2).
\end{align*}
Therefore, on simplifying further and equating real parts, we have, as $p\rightarrow0$,
\[\mathrm{Re}[\varphi_W(t)]=\frac{1}{1+t^2}+\frac{p\mu_2t^2(t^2-1)}{2\mu^2(1+t^2)^2}+O(p^2),\] 
and so the  $O(p)$ rate cannot be improved.
\end{proof}

\section{Approximation of the waiting time distributions of M/G/1 and G/G/1 queues by the coupling approach}\label{sec4}

In this section, we apply Theorem \ref{thmross} to prove Theorems \ref{thmbound} and \ref{gg1thm}. 

\vspace{3mm}

\noindent{\bf{Proof of Theorem \ref{thmbound} via the coupling approach.}}  We establish an upper bound for $d_{\mathrm{W}}(\mathcal{L}(\widetilde{W}),\mathcal{L}(Z))$. Suppose that the queue is stationary.  The queue is empty with probability $1-\rho$.  Let $R_i$ be the residual service time of customer $i$.  Then, recall from Section \ref{sec2} that the random variable $W$ can be expressed as $W=\sum_{i=1}^N R_i,$ where $N\sim \mathrm{Geo}^0(1-\rho)$.  Since the $R_i$ are IID and independent of $N\sim \mathrm{Geo}^0(1-\rho)$, we are in the setting of Theorem \ref{thmross}.  Here, using standard formulas for the moments of $R_1$, we have
\[p=1-\rho, \quad
\mu=\mathbb{E}[R_1]=\frac{\mathbb{E}[S^2]}{2\mathbb{E}[S]}, \quad
\mu_2=\mathbb{E}[R_1^2]=\frac{\mathbb{E}[S^3]}{3\mathbb{E}[S]}.\]
Plugging these values into Theorem \ref{thmross} yields the desired bound:
\begin{equation}\label{comp}d_{\mathrm{W}}(\mathcal{L}(\widetilde{W}),\mathcal{L}(Z))\leq \frac{4\mathbb{E}[S^3]\mathbb{E}[S]}{3(\mathbb{E}[S^2])^2}\frac{1-\rho}{\rho}.
\end{equation}
The optimality of the $O(1-\rho)$ rate as $\rho\rightarrow1$ is guaranteed by Theorem \ref{thmross}.

We now deduce an upper bound on $d_{\mathrm{W}}(\mathcal{L}(\hat{W}),\mathcal{L}(Z))$ from (\ref{comp}).  Being a probability metric, the Wasserstein distance satisfies the triangle inequality, and so we have
\begin{align}d_{\mathrm{W}}(\mathcal{L}(\hat{W}), \mathcal{L}(Z))&\leq d_{\mathrm{W}}(\mathcal{L}(\hat{W}), \mathcal{L}(\rho Z))+ d_{\mathrm{W}}(\mathcal{L}(\rho Z), \mathcal{L}(Z))\nonumber \\
&= d_{\mathrm{W}}(\mathcal{L}(\rho \widetilde{W}), \mathcal{L}(\rho Z))+ d_{\mathrm{W}}(\mathcal{L}(\rho Z), \mathcal{L}(Z))\nonumber \\
\label{tri}&= \rho d_{\mathrm{W}}(\mathcal{L}( \widetilde{W}), \mathcal{L}( Z))+ d_{\mathrm{W}}(\mathcal{L}(\rho Z), \mathcal{L}(Z)).
\end{align}
We have already bounded $d_{\mathrm{W}}(\mathcal{L}( \widetilde{W}), \mathcal{L}( Z))$, so it suffices to compute $d_{\mathrm{W}}(\mathcal{L}(\rho Z), \mathcal{L}(Z))$.  Recalling the definition (\ref{wass2}) of Wasserstein distance, we have that
\begin{equation}\label{expwass}d_{\mathrm{W}}(\mathcal{L}(\rho Z), \mathcal{L}(Z))=\int_0^\infty\big[(1-\mathrm{e}^{-x/\rho})-(1-\mathrm{e}^{-x})\big]\,\mathrm{d}x=1-\rho.
\end{equation}
Substituting (\ref{comp}) and (\ref{expwass}) into (\ref{tri}) yields the upper bound in (\ref{mg1bound2}).

Finally, we establish the lower bound in (\ref{mg1bound2}).  Recall that $\mathbb{E}[\hat{W}]=\rho$.  Then, since $h(x)=x$ is in the class $\mathrm{Lip}(1)$, it follows that
\begin{equation*}d_{\mathrm{W}}(\mathcal{L}(\hat{W}),\mathcal{L}(Z))\geq |\mathbb{E}[\hat{W}]-\mathbb{E}[Z]|=1-\rho.
\end{equation*}
The proof is complete. \hfill $\Box$

\vspace{3mm}

\noindent {\bf{Proof of Theorem \ref{gg1thm}.} }
Recall from Section \ref{sec2} that the random variable $W$ can be represented as $W=\sum_{j=1}^\Lambda Y_j$, where $\Lambda\sim \mathrm{Geo}(1-\eta)$.  Since the $Y_j$ are IID and independent of $\Lambda$, we are in the setting of Theorem \ref{thmross}.  Here we have
\[p=1-\eta, \quad \mu=\mathbb{E}[Y_1]=\mathbb{E}[\Sigma_{L(1)}], \quad \mu_2=\mathbb{E}[Y_1^2]=\mathbb{E}[\Sigma_{L(1)}^2].\]
Plugging these values into Theorem \ref{thmross} yields the bound (\ref{gg1bound}).  The optimality of the $O(1-\eta)$ rate as $\eta\rightarrow1$ is guaranteed by Theorem \ref{thmross}.
\hfill $\Box$	

\section{Approximation of the waiting time distributions of the M/G/1 queue by the generator approach}\label{sec5}

In this section, we prove Theorem \ref{thmbound} using the comparison of generators approach to Stein's method.  This approach was used in a series of papers of Braverman, Dai and Feng \cite{braverman18, bj1, bj2, bjf} to derive diffusion approximations for the number of customers in various queueing systems in steady state.  However, the approach applies in many other settings; see, for example, Ley, Reinert and Swan \cite{ley} in which the approach is used to bound the distance between standard probability distributions with respect to a probability metric.

\vspace{3mm}


\noindent{\bf{Proof of Theorem \ref{thmbound} via comparison of generators.}} We establish inequality (\ref{mg1bound}); the double inequality (\ref{mg1bound2}) then follows from exactly the same argument as was used in the coupling approach proof of Section \ref{sec4}.  We do not prove the assertion that the $O(1-\rho)$ rate of convergence is optimal.

  Let us first recall that the generator of $\widetilde{W}_t$ is given by
\begin{equation}\label{wgen1}G_{\widetilde{W}_t}f(x)=\lambda\int_0^\infty[f(x+\delta s)-f(x)]\,\mathrm{d}F(s)-\delta f'(x)\mathbf{1}(x>0),
\end{equation}
where $\delta=\frac{2\mathbb{E}[S]}{\mathbb{E}[S^2]}\frac{1-\rho}{\rho}$.  Now, let $G_Zf(x)$ be the left-hand side of the Stein equation (\ref{expsteineqn}):
\begin{equation*}G_Zf(x):=f''(x)-f'(x)+f'(0).
\end{equation*}
Suppose $h$ is Lipschitz.  Then the solution $f_h$, as given by (\ref{steinsoln}), of the $\mathrm{Exp}(1)$ Stein equation (\ref{expsteineqn}) satisfies the assumptions such that equation (\ref{genegn9}) holds.  Therefore from (\ref{eqnh}) and the fact that $\mathbb{E}[G_{\widetilde{W}_t}f_h(\widetilde{W})]=0$, we see that, for any $a>0$,
\begin{align}|\mathbb{E}[h(\widetilde{W})]-\mathbb{E}[h(Z)]|&=|\mathbb{E}[G_Zf_h(\widetilde{W})]|\nonumber \\
&=|a\mathbb{E}[G_{\widetilde{W}_t}f_h(\widetilde{W})]-\mathbb{E}[G_Zf_h(\widetilde{W})]| \nonumber \\
\label{compgen}&\leq\mathbb{E}|aG_{\widetilde{W}_t}f_h(\widetilde{W})-G_Zf_h(\widetilde{W})|.
\end{align}
To bound the right hand-side of (\ref{compgen}), we study the difference $aG_{\widetilde{W}_t}f_h(x)-G_Zf_h(x)$, where we will later select $a=\frac{2}{\lambda\delta^2\mathbb{E}[S^2]}$.  For that we perform a Taylor expansion on $G_{\widetilde{W}_t}f_h(x)$.  To this end, we note that
\begin{align*}f_h(x+\delta s)-f_h(x)=\delta sf_h'(x)+\frac{1}{2}\delta^2s^2f_h''(x)+\frac{1}{6}s^3\delta^3 f_h^{(3)}(\eta),
\end{align*}
where $\eta\in(x,x+\delta s)$.  Substituting into (\ref{wgen1}) and using the solution to the Stein equation that satisfies $f_h'(0)=0$, \eqref{steinsoln}, gives 
\begin{align}G_{\widetilde{W}_t}f_h(x)&=\lambda\int_0^\infty[\delta sf_h'(x)+\frac{1}{2}\delta^2s^2f_h''(x)]\,\mathrm{d}F(s)-\delta f_h'(x)+\delta f_h'(x)\mathbf{1}(x=0)+R\nonumber \\
&=\lambda\delta\mathbb{E}[S]f_h'(x)+\frac{1}{2}\lambda\delta^2\mathbb{E}[S^2]f_h''(x)-\delta f_h'(x)+R\nonumber \\
\label{geneq5}&=\frac{1}{2}\lambda\delta^2\mathbb{E}[S^2]\big(f_h''(x)-f_h'(x)\big)+R,
\end{align}
where
\begin{align*}|R|&=\frac{\lambda\delta^3}{6}\bigg|\int_0^\infty s^3f_h^{(3)}(\eta)\,\mathrm{d}F(s)\bigg|\leq \frac{\lambda\delta^3\|f_h^{(3)}\|_\infty}{6}\int_0^\infty s^3\,\mathrm{d}F(s)\leq \frac{\lambda\delta^3}{3}\|h'\|_\infty\mathbb{E}[S^3],
\end{align*}
and we used (\ref{fbound1}) to obtain the final inequality.  In obtaining (\ref{geneq5}) we used that
\[1-\lambda\mathbb{E}[S]=1-\rho=\frac{\rho\delta\mathbb{E}[S^2]}{2\mathbb{E}[S]}=\frac{1}{2}\lambda\delta^2\mathbb{E}[S^2].\]
Multiplying by the constant $\frac{2}{\lambda\delta^2\mathbb{E}[S^2]}$ now gives
\begin{align*}\frac{2}{\lambda\delta^2\mathbb{E}[S^2]}G_{\widetilde{W}_t}f_h(x)=f_h''(x)-f_h'(x)+\frac{2}{\lambda\delta\mathbb{E}[S^2]}R,
\end{align*}
which we recognise as the generator $G_Zf(x)$ with an additional error term. From (\ref{compgen}) and setting $a=\frac{2}{\lambda\delta^2\mathbb{E}[S^2]}$,  we have that
\begin{align*}|\mathbb{E}[h(\widetilde{W})]-\mathbb{E}[h(Z)]|&= \bigg|\frac{2}{\lambda\delta^2\mathbb{E}[S^2]}\mathbb{E}[G_{\widetilde{W}_t}f_h(\widetilde{W})]-\mathbb{E}[G_Zf_h(\widetilde{W})]\bigg|\\
&\leq \frac{2}{\lambda\delta^2\mathbb{E}[S^2]}\cdot\frac{\lambda\delta^3}{3}\|h'\|_\infty\mathbb{E}[S^3] =\|h'\|_\infty\frac{4\mathbb{E}[S]\mathbb{E}[S^3]}{3(\mathbb{E}[S^2])^2}\frac{1-\rho}{\rho},
\end{align*}
whence on setting $\|h'\|_\infty=1$ yields the Wasserstein distance bound (\ref{mg1bound}), as required. \hfill $\Box$

\subsection*{Acknowledgements}
RG is supported by a Dame Kathleen Ollerenshaw Research Fellowship, and acknowledges support from the grant COST-STSM-CA15109-34568.

\end{document}